\RequirePackage{etex}
\RequirePackage{easymat}

\documentclass[12pt]{elsarticle}
\usepackage{amsmath, amsthm,
amssymb,mathdots,extsizes}
\usepackage[all]{xy}
\usepackage[active]{srcltx}

\newcommand{\ff}{\mathbb F}

\newcommand{\mc}{\mathcal}

\newcommand{\matt}[1]{\left[\begin{smallmatrix}
 #1 \end{smallmatrix}\right]}
\newcommand{\mat}[1]{\begin{bmatrix}
 #1  \end{bmatrix}}
\newcommand{\arr}[1]{\left[\begin{array}
 #1\end{array}\right]}

\newtheorem{theorem}{Theorem}
\newtheorem{lemma}{Lemma}
\theoremstyle{remark}
\newtheorem{remark}{Remark}

\begin{document}

\title{Reduction of a pair of skew-symmetric matrices to its canonical form under congruence}

\author[bov]{Victor A.~Bovdi\corref{cor}} \ead{vbovdi@gmail.com,  v.bodi@uaeu.ac.ae}
\address[bov]{United Arab Emirates University, Al Ain, UAE}

\author[ser]{Tatiana G.~Gerasimova} \ead{gerasimova@imath.kiev.ua}

\author[bov]{Mohamed A.~Salim}
\ead{msalim@uaeu.ac.ae}

\author[ser]{Vladimir~V.~Sergeichuk}
\ead{sergeich@imath.kiev.ua}
\address[ser]{Institute of Mathematics, Tereshchenkivska 3,
Kiev, Ukraine}

\cortext[cor]{Corresponding author.}

\begin{abstract}
Let $(A,B)$ be a pair of skew-symmetric matrices over a field of characteristic not 2. Its regularization decomposition is a direct sum
\[
(\underline{\underline A},\underline{\underline B})\oplus (A_1,B_1)\oplus\dots\oplus(A_t,B_t)
\]
that is congruent to $(A,B)$,
in which $(\underline{\underline A},\underline{\underline B})$ is a pair of nonsingular matrices and $(A_1,B_1),\dots,(A_t,B_t)$ are singular indecomposable canonical pairs of skew-symmetric matrices under congruence.
We give an algorithm that constructs a regularization decomposition. We also give a constructive proof of the known canonical form of $(A,B)$ under  congruence over an algebraically closed field of characteristic not 2.
\end{abstract}


\maketitle

\section{Introduction}

We give an algorithm that for each pair of skew-symmetric matrices constructs its regularization decomposition.

Two pairs $(A,B)$ and $(A',B')$ of square matrices of the same size are \emph{congruent} if there exists a nonsingular matrix $S$ such that
\[
S(A,B)S^T:=(SAS^T,SBS^T)=(A',B').
\]
A \emph{direct sum} of pairs $(A,B)$ and $(A',B')$ is the pair
\[
(A,B)\oplus(A',B'):=\left(\mat{A&0\\0&A'}, \mat{B&0\\0&B'}\right).
\]
A \emph{regularizing decomposition} of a pair $(A,B)$ of skew-symmetric matrices over a field of characteristic not 2 is a direct sum
\begin{equation}\label{hjw}
(\underline{\underline A},\underline{\underline B})\oplus (A_1,B_1)\oplus\dots\oplus(A_t,B_t)
\end{equation}
that is congruent to $(A,B)$,
in which $(\underline{\underline A},\underline{\underline B})$ is a pair of nonsingular matrices of the same size and each $(A_i,B_i)$ is one of the pairs
\begin{align}\label{a1}
\mc J_n:=&
\left(\mat{0&I_n\\-I_n&0},\mat{0&J_n(0) \\-J_n(0)^T &0}\right),\\
\nonumber
\mc K_n:=&\left(\mat{0&J_n(0) \\-J_n(0)^T &0},\mat{0&I_n\\-I_n&0}\right),\\
\label{a3}
\mc L_n:=&\left(\mat{0&L_n\\-L_n^T&0},\mat{0&R_n \\-R_n^T &0}\right),\quad n=1,2,\dots,
\end{align}
where $J_n(0)$ is the $n\times n$ singular Jordan block and
\begin{equation}\label{sdw}
 L_n:=\mat{1&0&&0\\&\ddots&\ddots\\0&&1&0},\quad
 R_n:=\mat{0&1&&0\\&\ddots&\ddots\\0&&0&1}\qquad
((n-1)\text{-by-}n).
\end{equation}
In particular, $\mc L_1=([0],[0])$.
The canonical form of $(A,B)$ under congruence (see \eqref{a4}) ensures that
$(\underline{\underline A},\underline{\underline B})$---the \emph{regular part} of $(A,B)$---is determined up to congruence, and $(A_1,B_1),\dots,(A_t,B_t)$---the \emph{singular summands}---are determined uniquely up to permutations.

In Section \ref{ss2}, we give a \emph{regularization algorithm} that uses elementary transformations of matrices and for each pair of skew-symmetric matrices over a field of characteristic not 2 constructs its regularization decomposition under congruence.
Regularization algorithms were constructed for matrix pencils by Van Dooren \cite{van_door}, for cycles of linear mappings by Sergeichuk \cite{ser_cycl} and Varga \cite{var}, and for square matrices under congruence and *congruence by Horn and Sergeichuk \cite{horn-ser}.

The regularization decomposition \eqref{hjw} is  the first step towards the reduction of $(A,B)$ to its canonical form under congruence (see Theorem \ref{ttt} in Section \ref{ssw}): each pair of skew-symmetric matrices over an algebraically closed field $\ff$ of characteristic not 2 is congruent to a direct sum, determined uniquely up to permutation of summands, of pairs of the form
\begin{equation}\label{a4}
\mc J_n(\lambda) :=
\left(\mat{0&I_n\\-I_n&0},\mat{0&J_n(\lambda) \\-J_n(\lambda)^T &0}\right)\ (\lambda \in \ff),
\quad \mc K_n,\quad \mc L_n,
\end{equation}
where
\begin{equation*}\label{jjy}
J_n(\lambda ):=\mat{\lambda &&&0\\1&\lambda\\[-2mm] &\ddots&\ddots\\[-2mm]0&&1&\lambda}\qquad
(n\text{-by-}n).
\end{equation*}
If $\ff$ is not algebraically closed, then $J_n(\lambda )$ in \eqref{a4}  is replaced by any  indecomposable canonical matrix for similarity; for example, $J_n(\lambda )$
can be replaced by a
Frobenius block
\begin{equation*}\label{3a}
\begin{bmatrix} 0&&
0&-c_n\\[-2mm]1&\ddots&&\vdots
\\[-3mm]&\ddots&0&-c_2\\
0&&1& -c_1 \end{bmatrix},
\end{equation*}
in which $p(x)^\ell=x^n+c_1
x^{n-1}+\dots+ c_n$ is an
integer power of a polynomial
$p(x)$ that is irreducible over
$\mathbb F$. This canonical form of pairs of skew-symmetric matrices under congruence was given by Scharlau \cite{sch} in terms of Kronecker's modules; see also \cite{ser1,thom}.

In Section \ref{ssw}, we give another proof of the canonical form of a pair of skew-symmetric matrices over an algebraically closed field based on the regularization algorithm from Section \ref{ss2}.

Dmytryshyn and K\r{a}gstr\"{o}m \cite{dmy,dm-k} construct miniversal deformations of a pair of skew-symmetric matrices $(A,B)$ under congruence and study how small perturbations of $A$ and $B$ change the canonical form of  $(A,B)$ under congruence.

\section{Regularization algorithm for a pair of skew-symmetric matrices}\label{ss2}

We consider only matrix pairs in which both the matrices have the same size. All transformations that we make with matrix pairs in this section are congruence transformations. Thus,
\begin{equation}\label{ksw}
\parbox[c]{0.8\textwidth}{when we write that we make an elementary transformation of rows (columns) of one matrix from a pair,  it means that we also make the same elementary transformation  of rows (respectively, columns) of the other matrix, and then the same elementary transformation of columns (respectively, rows) of both matrices.
}
\end{equation}

A \emph{semi-regularization decomposition} of a pair $(A,B)$ of skew-symmetric matrices is a direct sum
\begin{equation*}\label{hjw1}
({\underline A},{\underline B})\oplus (A_1,B_1)\oplus\dots\oplus(A_t,B_t)
\end{equation*}
that is congruent to $(A,B)$,
in which $\underline A$ is a nonsingular matrix and each $(A_i,B_i)$ is of the form $\mc J_n$ or $\mc L_n$ (see \eqref{a1} and \eqref{a3}).

In this section, we give an algorithm that constructs a regularization decomposition of a pair $(A,B)$ of skew-symmetric matrices over a field $\ff$ of characteristic not 2. For this purpose, it is enough to give an algorithm that constructs a semi-regularization decomposition since if \eqref{hjw} is a semi-regularization decomposition of $(A,B)$ and
\[
(\underline{\underline B},\underline{\underline A})\oplus (B'_1,A'_1)\oplus\dots\oplus(B'_s,A'_s)
\]
is a semi-regularization decomposition of $(\underline B,\underline A)$ (and hence each $(B'_i,A'_i)$ is of the form $\mc J_n$), then
\[
(\underline{\underline A},\underline{\underline B})\oplus (A'_1,B'_1)\oplus\dots\oplus(A'_s,B'_s)
\oplus (A_1,B_1)\oplus\dots\oplus(A_t,B_t)
\]
is a regularization decomposition of $(A,B)$.

We suppose that $A$ has been reduced to its canonical form for congruence; that is,
\begin{equation}\label{fey}
(A,B)=\left(\arr{{c|c|c} 0&I&0\\\hline -I&0&0\\\hline0&0&0},
\arr{{c|c|c}
B_{11}&B_{12}&B_{13}\\\hline B_{21}&B_{22}&B_{23}\\\hline
B_{21}&B_{22}&B_{23}}
\right)
\end{equation}
and further
we use only those congruence transformations $S(A,B)S^T$ that preserve $A$; i.e, for which $SAS^T=A$.

For example, we can take $S=R\oplus R^{-T}\oplus I$, in which $R$ is a nonsingular matrix. If $R$ is an elementary matrix, then we obtain transformation (i) from the next paragraph. For example, we can add row $i1$ multiplied by $a\in\ff$ to row $j1$ (``row $il$''  means the $i$th row of $l$th horizontal strip in \eqref{fey}) and make the same transformation of columns. These transformations spoil blocks $(1,2)$ and $(2,1)$ of $A$; we restore them by subtracting column $j2$ multiplied by $a$ from column $i2$ and making the same transformation of rows.

The following row transformations that are coupled with the same column transformations (see \eqref{ksw}) do not change $A$:
\begin{itemize}
  \item[(i)] \begin{itemize}
               \item
An elementary row transformation in the first horizontal strip and the inverse row transformation in the second horizontal strip.
\item
An elementary row transformation in the third horizontal strip.
\end{itemize}

  \item[(ii)]
\begin{itemize}
  \item
Add row $i1$ multiplied by $a\in\ff$ to row $j2$ with $j\ne i$, then add row $j1$ multiplied by $a$ to row $i2$.

  \item
 Add row $i1$ multiplied by $a\in\ff$ to row $i2$.
\end{itemize}

  \item[(iii)]
\begin{itemize}
  \item
Add row $i2$ multiplied by $a\in\ff$ to row $j1$ with $j\ne i$, then add row $j2$ multiplied by $a$ to row $i1$.
  \item
Add row $i2$ multiplied by $a\in\ff$ to row $i1$.
\end{itemize}

  \item[(iv)] Multiply row $i1$ by $-1$, then interchange it with row $i2$.

   \item[(v)] Add row $i3$ multiplied by $a\in\ff$ to a row in strip 2 or 3.
\end{itemize}

In each step of the following algorithm, we reduce $(A,B)$ of the form \eqref{fey} by transformations (i)--(v) to a direct sum, in which some of direct summands are of the form $\mc K_n$ or $\mc L_n$, and delete these summands. The algorithm stops when we obtain a pair $(\underline A,\underline B)$ with a nonsingular $\underline A$.

\subsection*{Semi-regularization algorithm for the pair \eqref{fey}:}
\begin{enumerate}
  \item
If $B_{33}\ne 0$, we reduce it by transformations (i) to the form
\[
B_{33}=\mat{0&I_k&0\\-I_k&0&0\\0&0&0}\ (k>0),
\]
then using  transformations (v) we make zero all entries outside of $I_k$ that are located in the rows and columns crossing $I_k$ (due to \eqref{ksw}, we also  make zero all entries outside of $-I_k$ that are located in the rows and columns crossing $-I_k$). Delete \emph{$k$ direct summands $\mc K_1=(\matt{0&0\\0&0},\matt{0&1\\-1&0})$} from $(A,B)$ (thus, we delete the rows and columns that cross $I_k$ and the rows and columns that cross $-I_k$). We obtain $(A,B)$ with $B_{33}=0$.

 \item
 If the columns of vertical strip 3 of $B$ are linearly dependent, then we fix a maximal system of linearly independent columns, and make zero the other columns in vertical strip 3; they give \emph{direct summands $\mc L_1=([0],[0])$.} Deleting them, we obtain $(A,B)$ in which the columns of vertical strip 3 of $B$ are linearly independent.

   \item 

If vertical strip 3 of $B$ is empty, then $A$ is nonsingular and a semi-regularization decomposition has been constructed. Suppose that vertical strip 3 of $B$ is nonempty. If the last column of $B_{13}$ is zero, we make it nonzero by transformations (iv). Reduce the last column of $B_{13}$ to the form $[0\,\dots\,0\,1]^T$, then make zero all entries of the last column of $B$ under $1$. Make the last row of horizontal strip 1 of $B$ equaling $[0\,\dots\,0\, 1]$ and obtain
\[
B=\arr{{cc|c|cc}&\vdots&&&\vdots\\
\dots&.&\dots&\dots&1\\\hline
&\vdots&&&\vdots\\\hline
&\vdots&&0&\vdots\\
\dots&-1&\dots&\dots&.\\
},
\]
in which the dots denote zero entries. Two cases are possible:

\begin{itemize}
   \item[(a)]
 First suppose that the last row of $B_{23}$ is nonzero. We reduce it to $[0\,\dots\,0\, 1\,0]$, then make zero  all entries of $B$ above 1, reduce the last row of horizontal strip 2 to  $[0\,\dots\,0\, 1]$ and obtain
\[
B=\arr{{cc|cc|ccc}
&\vdots  &   &\vdots   &  &\vdots &\vdots \\
\dots&. &   \dots&. &.\,.  &0&1
                               \\ \hline   
&\vdots  &   &\vdots   &  &\vdots &\vdots \\
\dots&. &   \dots&. &.\,.  &1&0
                                \\ \hline   
&:  &   &:\\
\dots&0 &   \dots&-1 &  &0&\\
\dots&-1   &   \dots&0 &  &&\\
}
\]
in which the dots denote zero entries.
Thus, $(A,B)$ has the direct summand
\begin{equation}\label{mmn}
\left(
 \begin{MAT}(c){ccccl}
 &\phantom{=}&&    &\\
&1&&    &\scriptstyle \it 2  \\
-1&&&   &\scriptstyle \it 3\\
\phantom{1}&&&    &\scriptstyle \it 1\\
\phantom{1}&&&    &\scriptstyle \it 4\\
\scriptstyle \it  2&\scriptstyle \it 3&\scriptstyle \it 1 &\scriptstyle \it 4&
\addpath{(0,1,4)rrrruuuulllldddd}
\addpath{(1,1,3)uuuu}
\addpath{(2,1,3)uuuu}
\addpath{(0,3,3)rrrr}
\addpath{(0,4,3)rrrr}
\\
\end{MAT}\!\!\!\!,\:
 \begin{MAT}(c){ccccl}
  &\phantom{=}&&    &\\
&&&1    &\scriptstyle \it 2  \\
&&1&   &\scriptstyle \it 3\\
&-1&&    &\scriptstyle \it 1\\
-1&&&    &\scriptstyle \it 4\\
\scriptstyle \it  2&\scriptstyle \it 3&\scriptstyle \it 1 &\scriptstyle \it 4&
\addpath{(0,1,4)rrrruuuulllldddd}
\addpath{(1,1,3)uuuu}
\addpath{(2,1,3)uuuu}
\addpath{(0,3,3)rrrr}
\addpath{(0,4,3)rrrr}
\\
\end{MAT}\!\!\!\!\!
\right),
\end{equation}
in which all unspecified entries are zero.
Rearranging the rows and columns as indicated, we obtain
\[
\left(
 \begin{MAT}(c){ccccl}
 &\phantom{=}&&    &\\
\phantom{1}&&&    &\scriptstyle \it 1  \\
&&1&   &\scriptstyle \it 2\\
&-1&&    &\scriptstyle \it 3\\
\phantom{1}&&&    &\scriptstyle \it 4\\
\scriptstyle \it  1&\scriptstyle \it 2&\scriptstyle \it 3 &\scriptstyle \it 4&
\addpath{(0,1,4)rrrruuuulllldddd}
\addpath{(2,1,3)uuuu}
\addpath{(0,3,3)rrrr}
\\
\end{MAT}\!\!\!\!,\:
 \begin{MAT}(c){ccccl}
 &\phantom{=}&&    &\\
&&-1&    &\scriptstyle \it 1  \\
&&&1   &\scriptstyle \it 2\\
1&&&    &\scriptstyle \it 3\\
&-1&&    &\scriptstyle \it 4\\
\scriptstyle \it  1&\scriptstyle \it 2&\scriptstyle \it 3&\scriptstyle \it 4&
\addpath{(0,1,4)rrrruuuulllldddd}
\addpath{(2,1,3)uuuu}
\addpath{(0,3,3)rrrr}
\\
\end{MAT}\!\!\!\!\!
\right).
\]
Hence, \eqref{mmn} \emph{is congruent  to
$\mc K_2$.}
We delete the summand \eqref{mmn} from $(A,B)$ and repeat step 3.

   \item[(b)]
 Now suppose that the last row of $B_{23}$ is zero.
Then we repeat step 3, ignoring the last rows and columns of all strips of $A$ and $B$, and obtain
\[
\setlength{\arraycolsep}{3pt}
B=\arr{{ccc|ccc|ccc}
&:&: &    &&   &   &:&:
       \\
\dots&. &.   &\dots&\!\!.\,.&\!\!.\,.&\dots &1&0
       \\
\dots&. &.   &\dots&\!\!.\,.&\!\!.\,.&\dots &0&1
                  \\ \hline   
& :&:  &    &&   &   &:&:
       \\
&\smash{\vdots}& \smash{\vdots} && & & &0&0
       \\
&\smash{\vdots}& \smash{\vdots}&& & &\!\!0 &\!\!\!\dots\!&0
                  \\ \hline   
&: &: &    &&0 &   &&
       \\
\dots&-1 &0   &\dots&0 & \smash{\vdots} &&0&
       \\
\dots&0 &-1   &\dots&0 &0 &&&
\\
}
\]

If the penultimate row of $B_{23}$ is nonzero, then $(A,B)$ has a direct summand of the form \eqref{mmn} that \emph{is congruent  to $\mc K_2$}, we delete it and repeat step (b).

 \end{itemize}

We repeat (a) and (b) until we obtain
\begin{equation}\label{cdo}
\setlength{\arraycolsep}{4.5pt}
(A,B)=
\left(
\arr{{cc|cc|c} 0&0 & I&0&0\\
0&0 & 0&I&0\\
\hline
-I&0 &0&0&0\\
0&-I&0&0&0\\
\hline
0&0& 0&0&0\\
},
\arr{{cc|cc|c} B_{11}'&0 & B_{12}'&B_{13}'&0\\
0&0 & 0&0&I\\
\hline
B_{21}'&0 & B_{22}'&B_{23}'&0\\
B_{31}'&0 & B_{32}'&B_{33}'&0\\
\hline
0&-I & 0&0&0\\
}\right),
\end{equation}

 \item
Let us reduce \eqref{cdo} by those congruence transformations that preserve $A$ and strips 2 and 5 of $B$,  both vertical and horizontal. Then its subpair
\begin{equation}\label{fer}
(A',B'):=\left(\mat{0&I&0\\ -I&0&0\\0&0&0},
\mat{
B_{11}'&B_{12}'&B_{13}'\\
B_{21}'&B_{22}'&B_{23}'\\
B_{31}'&B_{32}'&B_{33}'}
\right)
\end{equation}
 is reduced by those congruence transformations that preserve $A'$, which means that $B'$ is reduced by transformations (i)--(v).\footnote{
For example, we can add row $i$ of
$[B_{31}'\,B_{32}'\,B_{33}']$ to row $j$ of
$[B_{21}'\,B_{22}'\,B_{23}']$. This addition spoils the zero block $(3,2)$ of $A$ in \eqref{cdo}; we restore it by  subtracting column $j$ of vertical strip 1 from column $i$ of vertical strip 2. This addition spoils vertical strip 2 of $B$; we restore it by the rows of the last strip of $B$.}

We apply to  $(A',B')$  steps 1--4:
\begin{itemize}
  \item
If $B_{33}'\ne 0$, then we delete summands of the form
\begin{equation*}\label{mmf}
\left(
 \begin{MAT}(c){ccccccl}
 &\phantom{=}&&&&&\\
&&1&&    &&\scriptstyle \it 3  \\
&&&1&    &&\scriptstyle \it 4  \\
-1&&&&    &&\scriptstyle \it 5  \\
&-1&&&    &&\scriptstyle \it 2  \\
&&&&\phantom{-1}    &&\scriptstyle \it 1 \\
\phantom{-1}&&&&    &&\scriptstyle \it 6\\
\scriptstyle \it  3&\scriptstyle \it 4&\scriptstyle \it 5 &\scriptstyle \it 2&\scriptstyle \it  1&\scriptstyle \it 6&
\addpath{(0,1,4)rrrrrruuuuuulllllldddddd}
\addpath{(2,1,3)uuuuuu}
\addpath{(4,1,3)uuuuuu}
\addpath{(0,3,3)rrrrrr}
\addpath{(0,5,3)rrrrrr}
\\
\end{MAT}\!\!\!\!\!,\:
 \begin{MAT}(c){ccccccl}
&\phantom{=}&&&&&\\
&&&&    &1&\scriptstyle \it 3  \\
&&&&    1&&\scriptstyle \it 4  \\
&&&1&    &&\scriptstyle \it 5  \\
&&-1&&    &&\scriptstyle \it 2  \\
&-1&&&    &&\scriptstyle \it 1 \\
-1&&&&    &&\scriptstyle \it 6\\
\scriptstyle \it  3&\scriptstyle \it 4&\scriptstyle \it 5 &\scriptstyle \it 2&\scriptstyle \it  1&\scriptstyle \it 6&
\addpath{(0,1,4)rrrrrruuuuuulllllldddddd}
\addpath{(2,1,3)uuuuuu}
\addpath{(4,1,3)uuuuuu}
\addpath{(0,3,3)rrrrrr}
\addpath{(0,5,3)rrrrrr}
\\
\end{MAT}\!\!\!\!\right),
\end{equation*}
in which all unspecified entries are zero.
These summands \emph{are congruent to $\mc K_3$} since
rearranging the rows and columns as indicated, we obtain
\begin{equation*}\label{mrf}
\left(
 \begin{MAT}(c){ccccccl}
 &&&& \phantom{=}   &&\\
&&&&  \phantom{1}  &&\scriptstyle \it 1  \\
&&&-1&    &&\scriptstyle \it 2  \\
&&&&    1&&\scriptstyle \it 3 \\
&1&\phantom{-1}&&    &&\scriptstyle \it 4  \\
&&-1&&    &&\scriptstyle \it 5 \\
\phantom{-1}&&&&    &&\scriptstyle \it 6\\
\scriptstyle \it  1&\scriptstyle \it 2&\scriptstyle \it 3 &\scriptstyle \it 4&\scriptstyle \it  5&\scriptstyle \it 6&
\addpath{(0,1,4)rrrrrruuuuuulllllldddddd}
\addpath{(3,1,3)uuuuuu}
\addpath{(0,4,3)rrrrrr}
\\
\end{MAT}\!\!\!\!\!,\:
 \begin{MAT}(c){ccccccl}
 &&&& \phantom{=}   &&\\
&&&-1&    &&\scriptstyle \it 1  \\
&&&&    -1&&\scriptstyle \it 2 \\
&&&&    &1&\scriptstyle \it 3 \\
1&&&&    &&\scriptstyle \it 4  \\
&1&&&    &&\scriptstyle \it 5 \\
&&-1&&    &&\scriptstyle \it 6\\
\scriptstyle \it  1&\scriptstyle \it 2&\scriptstyle \it 3 &\scriptstyle \it 4&\scriptstyle \it  5&\scriptstyle \it 6&
\addpath{(0,1,4)rrrrrruuuuuulllllldddddd}
\addpath{(3,1,3)uuuuuu}
\addpath{(0,4,3)rrrrrr}
\\
\end{MAT}\!\!\!\!\!\right).
\end{equation*}
We have obtained $B_{33}'=0$.

  \item
 If the columns of vertical strip 3 of $B'$ are linearly dependent, then we delete from \eqref{cdo} the \emph{summands}
\begin{equation}\label{nnr}
\mc L_2=\left(\arr{{c|cc}0&1&0\\\hline -1&0&0\\0&0&0},
\arr{{c|cc}0&0&1\\\hline 0&0&0\\-1&0&0}
\right).
\end{equation}

  \item

Reasoning as in step 4, we delete from $(A,B)$ all summands that \emph{are congruent to $\mc K_3$} and obtain $(A,B)$ of the form
\[
\setlength{\arraycolsep}{3.5pt}
\left(
\arr{{ccc|ccc|c}
&& & I&0&0&\\
&0& & 0&I&0& 0\\
&& & 0&0&I& \\
\hline
-I&0 &0&&&&\\
0&-I&0 &&0&&0\\
0&0&-I & &&& \\
\hline
&0&& &0&&0\\
},
\arr{{ccc|ccc|c}
B_{11}''&0&0 & B_{12}''&B_{13}''&0&0\\
0&0&0 & 0&0&I& 0\\
0&0&0 & 0&0&0& I\\
\hline
B_{21}''&0 &0& B_{22}''&B_{23}''&0&0\\
B_{31}''&0&0 & B_{32}''&B_{33}''&0&0\\
0&-I&0 & 0&0&0& 0\\
\hline
0&0&-I & 0&0&0&0\\
}\right),
\]
in which
\[
(A'',B''):=\left(\mat{0&I&0\\ -I&0&0\\0&0&0},
\mat{
B_{11}''&B_{12}''&B_{13}''\\
B_{21}''&B_{22}''&B_{23}''\\
B_{31}''&B_{32}''&B_{33}''}
\right)
\]
is reduced by arbitrary congruence transformations that preserve $A''$, and so $B''$ is reduced by transformations (i)--(v).
\end{itemize}

We repeat this reduction until we obtain $(A^{(k)},B^{(k)})$, in which the third vertical and horizontal strips are empty. Then $A^{(k)}=\matt{0&I\\-I&0}$ and $(\underline A, \underline B):=(A^{(k)},B^{(k)})$.
\end{enumerate}

\begin{remark}\label{kah}
The regularization algorithm for pairs $(A,B)$ of skew-symmetric matrices constructed by Kozlov \cite[\S\,3]{koz} has a gap. Steps 1 and 2 of his algorithm reduce $(A,B)$ to the form
\begin{equation*}\label{mms}
(A',B'):=\left(
\mat{0&-1&0&\cdots&0\\
1&0&0&\cdots&0\\
0&0&*&\cdots&*\\[-5pt]
\vdots&\vdots&\vdots&&\vdots\\
0&0&*&\cdots&*}
                                     ,\
\mat{0&0&0&\cdots&0\\
0&*&*&\cdots&*\\
0&*&*&\cdots&*\\[-5pt]
\vdots&\vdots&\vdots&&\vdots\\
0&*&*&\cdots&*}
\right).
\end{equation*}

He states that step 3 reduces $B'$ to the form
\[
\mat{
0&0&0&0&\cdots&0\\
0&0&1&0&\cdots&0\\
0&-1&0&0&\cdots&0\\
0&0&0&*&\cdots&*\\[-5pt]
\vdots&\vdots&\vdots&\vdots&&\vdots\\
0&0&0&*&\cdots&*}
\]
while
preserving $A'$,
which is impossible: under there transformations with $B'$ the first row of $A'$ is reduced to the form $[0\,-1\,*\,\cdots\,*]$. \emph{The canonical form of pairs of skew-symmetric matrices cannot be proved in an elementary way} (as in \cite{koz}) since it contains Kronecker's canonical form for matrix pencils.
\end{remark}

\section{Proof of the canonical form for pairs over an algebraically closed field}\label{ssw}

In this section, we prove the following well-known  theorem (see \cite{sch,ser1,thom})
using the regularization algorithm from Section \ref{ss2} and the method that was developed by Nazarova and Roiter \cite{naz+roi} (see also \cite{naz+roi+ser} and \cite[Sect. 1.8]{gab-roi}) to prove Kronecker's canonical form for matrix pencils.

\begin{theorem}\label{ttt}
Each pair of skew-symmetric matrices over  an algebraically closed field\/ $\ff$ of characteristic not $2$ is congruent to a direct sum of pairs of the form
\[
\mc J_n(\lambda)\ (\lambda \in\ff),\quad
\mc K_n,\quad \mc L_n
\]
$($see \eqref{a4}$)$.
This sum is uniquely determined, up to permutations of summands.
\end{theorem}

A pair of skew-symmetric matrices is \emph{indecomposable} if it is not congruent to a direct sum of pairs of skew-symmetric matrices of smaller sizes.
The algorithm from Section \ref{ss2} is used only in the proof of  the following lemma.

\begin{lemma}\label{jyp}
Over a field of characteristic not $2$, let $(A,B)$ be an indecomposable pair of skew-symmetric matrices, in which $A$  singular. Then $(A,B)$ is congruent to a pair
satisfying the following condition:
\begin{equation}\label{rly}
\parbox[c]{0.8\textwidth}{each row and  each column of its matrices contains at most one nonzero entry that is $1$ or $-1$ and the other entries are zero.}
\end{equation}
\end{lemma}

\begin{proof}
We apply steps 1--3 of the semi-regularization algorithm from Section \ref{ss2} and obtain that $(A,B)$ is congruent to a direct sum of \eqref{cdo} and pairs of the form $\mc  K_1$, $\mc  L_1$, and $\mc  K_2$. Since $(A,B)$ is indecomposable, there are two possibilities:
\begin{itemize}
  \item $(A,B)$ is congruent to $\mc  K_1$, $\mc  L_1$, or $\mc  K_2$; they satisfy \eqref{rly}.

  \item $(A,B)$ is congruent to \eqref{cdo}.
Since $A$ is singular, the size of  $(A',B')$ is less than the size of  $(A,B)$.  Reasoning by induction, we suppose that $(A',B')$ defined in \eqref{fer} satisfies \eqref{rly}.  Then $(A,B)$ satisfies \eqref{rly} too, which follows from the form of \eqref{cdo}.
\end{itemize}
\vskip-2em
 \end{proof}

\begin{proof}[Proof of Theorem \ref{ttt}]
Let $(A,B)$ be an indecomposable pair of skew-symmetric matrices over $\ff$.  Two cases are possible.
\medskip

\emph{Case 1:  $A$ is singular.} By congruence transformations with $(A,B)$, we make $A$ and $B$ satisfying \eqref{rly}. Let us show that the obtained $(A,B)$ is reduced to $\mc K_n$ or $\mc L_n$ by the following congruence transformations:
\begin{itemize}
  \item
  interchange  rows $i$ and $j$  and then interchange columns $i$ and $j$ in both the matrices,

  \item
  multiply row $i$ and column $i$ by $-1$ in both the matrices.
\end{itemize}
We make only these transformations with $(A,B)$. For example (as in \eqref{ksw}), if we write about interchanging two columns in $A$, keep in mind that we make the same interchange of columns in $B$ and rows in $A$ and $B$.

Let $A$ and $B$ be $(2n-1)\times (2n-1)$ or $2n\times 2n$; partition them into blocks
\[
(A,B)=\left(\arr{{c|c}A_{11}&A_{12}\\\hline A_{21}&A_{22}},
\arr{{c|c}B_{11}&B_{12}\\\hline B_{21}&B_{22}}\right),\quad
\text{$A_{22}$ and $B_{22}$ are $n\times n$.}
\]

Since $A$ is singular and satisfies \eqref{rly}, it has a zero column; we interchange it with the last column and obtain a zero last column.
If the last column of $B$ is also zero, then $\mc L_1=([0],[0])$ is a direct summand of $(A,B)$. Since $(A,B)$ is indecomposable, $(A,B)=\mc L_1$, which proves the theorem in this case.

It remains to consider the case when the last column of $B$ is nonzero. Then  it contains $\varepsilon \in\{-1,1\}$. We multiply the row containing $\varepsilon $ by $\varepsilon$,  interchange it with the last row of $[B_{11}\,B_{12}]$, and obtain
\[
(A,B)=\left(\arr{{cc|cc}&&  &\vdots\\
&&&0\\\hline
&&  &\vdots\\
\dots&0&\dots&.},
\arr{{cc|cc}&\vdots&  &\vdots\\
\dots&.&\dots&1\\\hline
&\vdots&  &\vdots\\
\dots&-1&\dots&.}
\right),
\]
in which the dots denote zero entries.
If the last row of $[A_{11}\,A_{12}]$ is zero, then $\mc K_1=(\matt{0&0\\0&0},\matt{0&1\\-1&0})$ is a direct summand of $(A,B)$, and so $(A,B)=\mc K_1$, which proves the theorem in this case.

It remains to consider the case when the last row of $[A_{11}\,A_{12}]$  contains
$\varepsilon \in\{-1,1\}$. We multiply the column containing $\varepsilon $ by $\varepsilon$,  interchange it with the penultimate column, and obtain
\[
(A,B)=\left(\arr{{cc|ccc}&\vdots&  &\vdots&\vdots\\
\dots&.&\dots&1&0\\\hline
&\vdots&  &\vdots&\vdots\\
\dots&-1&\dots&.&. \\
\dots&0&\dots&.&. },
\arr{{cc|ccc}&\vdots&  &&\vdots\\
\dots&.&\dots&0&1\\\hline
&\vdots&  &&\vdots\\
&0&&&. \\
\dots&-1&\dots&.&. }
\right).
\]
If the penultimate column of $B$ is zero, then $\mc L_2$ (see \eqref{nnr}) is a direct summand of $(A,B)$, and so $(A,B)=\mc L_2$.

It remains to consider the case  when the penultimate column of $B$  contains
$\varepsilon \in\{-1,1\}$. We multiply the row containing  $\varepsilon $ by $\varepsilon$,  interchange it with the penultimate row of $[B_{11}\,B_{12}]$, and obtain
\[
(A,B)=\left(\arr{{ccc|ccc}&&\vdots&  &\vdots&\vdots\\
&&.&  &0&0\\
\dots&.&.&\dots&1&0\\\hline
&&\vdots&  &\vdots&\vdots\\
\dots&0&-1&\dots&.&.\\
\dots&0&0&\dots&.&. },
\arr{{ccc|ccc}&\vdots&\vdots&  &\vdots&\vdots\\
\dots&.&.&\dots&1&0\\
\dots&.&.&\dots&0&1\\\hline
&\vdots&\vdots&  &\vdots&\vdots\\
\dots&-1&0&\dots&.&.\\
\dots&0&-1&\dots&.&. }\right).
\]
If the  penultimate row of $[A_{11}\,A_{12}]$ is zero, then $\mc K_2$ is a direct summand of $(A,B)$, and so $(A,B)=\mc K_2$.

It remains to consider the case when the  penultimate row of $[A_{11}\,A_{12}]$ contains
$\varepsilon \in\{-1,1\}$. We multiply the column containing  $\varepsilon $ by $\varepsilon$,  interchange it with the pre-penultimate column, and obtain
\[
\left(\arr{{ccc|cccc}&\vdots&\vdots&  &\vdots&\vdots&\vdots\\
\dots&.&.&\dots&1  &0&0\\
\dots&.&.&\dots&0&1&0\\\hline
&\vdots&\vdots&  &\vdots&\vdots&\vdots\\
\dots&-1&0&\dots&.&. &.\\
\dots&0&-1&\dots&.&.&.\\
\dots&0&0&\dots&.&. &.},
\arr{{ccc|cccc}&\vdots&\vdots&  &&\vdots&\vdots\\
\dots&.&.&\dots&0  &1&0\\
\dots&.&.&\dots&0&0&1\\\hline
&\vdots&\vdots&  &&\vdots&\vdots\\
&0&0&&&. &.\\
\dots&-1&0&\dots&.&.&.\\
\dots&0&-1&\dots&.&. &.}\right),
\]
and so on.

Repeating this reduction, we find that $(A,B)$ is congruent to $\mc K_n$ or $\mc L_n$.
\medskip

\emph{Case 2: $A$ is nonsingular.} Since $\ff$ is algebraically closed, there exists $\lambda \in\ff$ such that $\det(A\lambda-B)=0$. By case 1, there is a nonsingular $S$ such that
\[S(B-\lambda A,A)S^T
=
\mc K_n=\left(\mat{0&J_n(0) \\-J_n(0)^T &0},\mat{0&I_n\\-I_n&0}\right).
\]
Then
\[S(B,A)S^T
=
\left(\mat{0&J_n(\lambda) \\-J_n(\lambda)^T &0},
\mat{0&I_n\\-I_n&0}\right),
\]
and so $(A,B)$ is congruent to $\mc J_n(\lambda)$.
\medskip

We have proved that each pair of skew-symmetric matrices over $\ff$ is congruent to a direct sum of pairs of the form  $\mc J_n(\lambda)$,  $\mc K_n$, and $\mc L_n$. Let us prove the uniqueness of this direct sum. Two pairs $(A,B)$ and $(A',B')$ of matrices of the same size are \emph{equivalent} if there exist nonsingular matrices $R$ and $S$ such that
$
R(A,B)S:=(RAS,RBS)=(A',B')$. Thus, congruent pairs are equivalent.
By Kronecker's theorem for matrix pencils (see \cite[Section XII]{gan2}), each
matrix pair over $\ff$ is equivalent to a direct sum, uniquely determined up to permutation of summands, of pairs of the types
\begin{equation}\label{rtx}
(I_n,J_n(\lambda)),\quad (J_n(0),I_n),\quad  (L_n,R_n),\quad  (L_n^T,R_n^T),
\end{equation}
in which $n\in\{1,2,\dots\}$, $\lambda \in\ff$, and $L_n$ and $R_n$ are defined in \eqref{sdw}.

The pairs  $\mc J_n(\lambda)$,  $\mc K_n$,  $\mc L_n$ are equivalent to
\begin{equation}\label{trx}
(I_n,J_n(\lambda))\oplus(I_n,J_n(\lambda)),\
(J_n(0),I_n)\oplus(J_n(0),I_n),\ (L_n,R_n)\oplus (L_n^T,R_n^T).
\end{equation}
Thus, two distinct direct sums of pairs of the form $\mc J_n(\lambda)$,  $\mc K_n$, and $\mc L_n$ (determined up to permutations of summands) have distinct canonical forms for equivalence, and so these sums cannot be congruent, which proves the uniqueness in Theorem \ref{ttt}.
\end{proof}

\begin{remark}
 Theorem \ref{ttt} can  also be proved by using the description of Kronecker's canonical forms for pairs of skew-symmetric matrices under equivalence and the following surprising statement from \cite[Corollary 35.2]{mac} (see also \cite[\S\,61 and \S\,62]{dic}) for matrix pairs over an algebraically closed field $\ff$ of characteristic not 2:
 \begin{equation}\label{uvp}
\parbox[c]{0.8\textwidth}{Let $(A,B)$ and $(A',B')$ be two matrix pairs, in which $A$ and $A'$, also $B$ and $B'$, are both symmetric or  both skew-symmetric.  Then $(A,B)$ and $(A',B')$ are congruent if and only if they are equivalent.
}
\end{equation}
This statement was generalized to pairs of skew-symmetric matrices over any field of characteristic zero by Williamson \cite{wil}.
An analogous statement for
arbitrary systems of forms and linear mappings is given in \cite{roi} and
\cite[Theorem 1 and \S \,2]{ser1}.

Let $(A,B)$ be a pair of skew-symmetric matrices. By Kronecker's theorem for matrix pencils (see \eqref{rtx}), $(A,B)$ is equivalent to a direct sum
\begin{equation}\label{nky}
\bigoplus_i (I_{m_i},J_{m_i}(\lambda _i))\oplus
\bigoplus_j (J_{n_j}(0),I_{n_j})\oplus
\bigoplus_k (L_{r_k},R_{r_k})\oplus
\bigoplus_l (L_{s_l}^T,R_{s_l}^T)
\end{equation}
that is determined uniquely up to permutations of summands.

The pair $(A,B)=(-A^T,-B^T)$ is equivalent to $(A^T,B^T)$, which is equivalent to
\begin{equation*}\label{nk1y}
\bigoplus_i (I_{m_i},J_{m_i}(\lambda _i))\oplus
\bigoplus_j (J_{n_j}(0),I_{n_j})\oplus
\bigoplus_k (L_{r_k}^T,R_{r_k}^T)\oplus
\bigoplus_l (L_{s_l},R_{s_l})
\end{equation*}
Since the sum \eqref{nky} is uniquely determined by $(A,B)$, up to permutations of direct summands, the summands  of the third and fourth types in \eqref{nky} occur in pairs
\begin{equation*}
\label{dfeu}
(L_{r_k},R_{r_k})\oplus
 (-L_{r_k}^T,-R_{r_k}^T)
 \end{equation*}
By  \cite[p. 335]{thom}, the summands  of the first and second types also occur in pairs
\begin{equation*}\label{dfe1}
(I_n,J_n(\lambda))\oplus(I_n,J_n(\lambda)),\qquad
(J_n(0),I_n)\oplus(J_n(0),I_n).
\end{equation*}
Since the pairs  $\mc J_n(\lambda),  \mc K_n,\mc L_n$ are equivalent to \eqref{trx},
$(A,B)$ is equivalent to a direct sum of pairs of the types $\mc J_n(\lambda),  \mc K_n,\mc L_n$, and this sum is uniquely determined, up to permutations of summands. By \eqref{uvp},  $(A,B)$ is congruent to this direct sum.

\section*{Acknowledgement}
The  work  was  supported  in  part  by  the  UAEU  UPAR  grants  G00001922  and
G00002160.

\end{remark}


\begin{thebibliography}{99}
\bibitem{dic}
L.E. Dickson, Modern Algebraic Theories,
Benj. H. Sanborn \& Co, Chicago, 1926. Republished as: Algebraic Theories, Dover Publications, Inc., New York, 1959.

\bibitem{dmy}
A. Dmytryshyn, Miniversal deformations of pairs of skew-symmetric matrices under congruence,
Linear Algebra Appl. 506 (2016) 506--534.

\bibitem{dm-k}
A. Dmytryshyn, B. K\r{a}gstr\"{o}m, Orbit closure hierarchies of skew-symmetric matrix pencils, SIAM J. Matrix Anal. Appl. 35 (2014) 1429--1443.

\bibitem{gab-roi}
P. Gabriel, A.V. Roiter, Representations of Finite-Dimensional Algebras, Springer-Verlag, 1997.

\bibitem{gan2}
F.R. Gantmacher,  The Theory of Matrices, Vol. 2, AMS Chelsea Publishing, Providence, RI, 1998.

\bibitem{horn-ser}
R.A. Horn, V.V. Sergeichuk, A regularization algorithm for matrices of bilinear and sesquilinear forms, Linear Algebra Appl. 412 (2006) 380--395.

\bibitem{koz}
I.K. Kozlov, An elementary proof of the Jordan--Kronecker theorem,
Math. Notes 94 (2013) 885--896.

\bibitem{mac}
C.C. Mac Duffee, The Theory of Matrices,
Verlag Von Julius Springer, Berlin, 1933.

\bibitem{naz+roi}
L.A. Nazarova, A.V. Roiter,
Finitely generated modules over
a dyad of two local Dedekind
rings, and finite groups which
possess an abelian normal
divisor of index $p$, Izv. Akad. Nauk
SSSR Ser. Mat. 33 (1969)
65--89 (in
Russian).

\bibitem{naz+roi+ser}
L.A. Nazarova, A.V. Roiter, V.V. Sergeichuk, V.M. Bondarenko, Application of modules over a dyad for the classification of finite $p$-groups possessing an abelian subgroup of index $p$ and of pairs of mutually annihilating operators, J. Soviet Math. 3 (5) (1975) 636--654.

\bibitem{roi} A.V. Roiter, Bocses with involution, in:
Representations and Quadratic Forms, Akad. Nauk Ukrain. SSR, Inst. Mat., Kiev,  1979, 124--128 (in Russian).

\bibitem{sch}
R. Scharlau,  Paare alternierender Formen, Math. Z. 147  (1976)  13--19.

\bibitem{ser1} V.V. Sergeichuk, Classification problems for system of forms
and linear mappings, Math. USSR-Izv. 31 (1988) 481--501.

\bibitem{ser_cycl}
V.V. Sergeichuk, Computation of canonical matrices for chains and cycles of linear mappings, Linear
Algebra Appl. 376 (2004) 235--263.


\bibitem{thom} R.C. Thompson, Pencils of complex and real symmetric and skew
matrices, Linear Algebra Appl. 147 (1991) 323--371.

\bibitem{van_door}
P. Van Dooren, The computation
of Kronecker's canonical form of
a singular pencil,  Linear
Algebra Appl.  27 (1979)
103--140.

\bibitem{var}
A. Varga, Computation of Kronecker-like forms of periodic matrix pairs,
in Proceedings of the 16th International Symposium on the Mathematical
Theory of Networks and Systems, Leuven, Belgium, 2004, pp. 5--9.

\bibitem{wil}
J. Williamson, The conjunctive equivalence of pencils of hermitian and anti-hermitian matrices, Amer. J. Math. 59 (1937) 399--413.

\end{thebibliography}
\end{document}